\documentclass[conference,final]{IEEEtran}
\IEEEoverridecommandlockouts
\usepackage{cite}
\usepackage{amsmath,amssymb,amsfonts, mathtools, amsthm}
\usepackage[dvipsnames]{xcolor}
\usepackage[
    colorlinks=true,
    linkcolor=blue,
    citecolor=blue,
    urlcolor=blue
]{hyperref}
\usepackage[nameinlink]{cleveref}
\usepackage{graphicx}
\usepackage{textcomp}

\usepackage{comment}

\usepackage[T1]{fontenc}
\usepackage[utf8]{inputenc}
\usepackage[english]{babel}

\newcommand{\minimize}{\textrm{minimize}}

\newcommand{\R}{\mathbb{R}}

\newcommand{\calO}{\mathcal{O}}

\definecolor{purple}{rgb}{0.74, 0.2, 0.64}

\newcommand{\D}{\mathrm{D}}

\newcommand{\transpose}{^\top}

\newcommand{\St}{\mathrm{St}}
\newcommand{\grad}{\mathrm{grad}}

\newcommand{\T}{^\top}

\newcommand{\rmT}{\mathrm{T}}

\newcommand{\rmN}{\mathrm{N}}
\newcommand{\Sym}{\mathrm{Sym}}

\renewcommand{\skew}{\operatorname{skew}}

\newcommand{\sigmamin}{\sigma_\mathrm{min}}

\newcommand{\Rnp}{{\mathbb{R}^{n \times p}}}
\newcommand{\Rnn}{{\mathbb{R}^{n \times n}}}

\usepackage{bm}

\newcommand{\range}{\operatorname{range}}

\usepackage{algorithm}
\usepackage{algpseudocode}
\newcommand{\inv}{^{-1}}

\newcommand{\aref}[1]{\hyperref[#1]{A\ref{#1}}}
\newcommand{\norm}[1]{\left\|#1\right\|}

\newcommand{\fronorm}[1]{\left\|#1\right\|_\mathrm{F}}

\newtheorem{theorem}			     {Theorem}	[section]
	
\newtheorem{lemma}	      [theorem]  {Lemma}		

\numberwithin{algorithm}{section}
\numberwithin{equation}{section}

\crefname{algorithm}{Algorithm}{Algorithms}
\crefname{lemma}{Lemma}{Lemmas}
\crefname{theorem}{Theorem}{Theorems}
\crefname{proposition}{Proposition}{Propositions}
\crefname{section}{Section}{Sections}
\crefname{figure}{Figure}{Figures}

\newcommand{\zetaprime}{{\zeta'}}

\newcommand{\StBnp}{\mathrm{St}^{B}(n,p)}

\newcommand{\StXBnp}{\mathrm{St}_{X}^{B}(n,p)}

\newcommand{\Rnk}{\mathbb{R}^{n\times k}}
\newcommand{\Rnpstar}{\mathbb{R}^{n\times p}_{*}}

\newcommand{\dNsketch}{d_{\mathrm{N}}^{\mathrm{sketch}}}
\newcommand{\dTsketch}{d_{\mathrm{T}}^{\mathrm{sketch}}}

\newcommand{\bbE}{\mathbb{E}}
\newcommand{\Stnp}{\St(n,p)}

\usepackage{tcolorbox}
\newcommand{\revision}{}

\begin{document}

\title{The sketched landing method for large-scale \\optimization under orthogonality constraints
\thanks{This work was supported by the Fonds de la Recherche Scientifique - FNRS under Grant no.~T.0001.23. Simon Mataigne is a Research Fellow of the Fonds de la Recherche Scientifique - FNRS.}
}

\author{\IEEEauthorblockN{Florentin Goyens}
\IEEEauthorblockA{\textit{ICTEAM Institute} \\
\textit{UCLouvain}\\
Louvain-la-Neuve, Belgium\\
florentin.goyens@uclouvain.be}
\and
\IEEEauthorblockN{Simon Mataigne}
\IEEEauthorblockA{\textit{ICTEAM Institute} \\
\textit{UCLouvain}\\
Louvain-la-Neuve, Belgium\\
simon.mataigne@uclouvain.be}
\and
\IEEEauthorblockN{P.-A. Absil}
\IEEEauthorblockA{\textit{ICTEAM Institute} \\
\textit{UCLouvain}\\
Louvain-la-Neuve, Belgium\\
pa.absil@uclouvain.be}
}

\maketitle

\begin{abstract}
 We propose the \emph{sketched landing method}, a randomized variant of the landing method for optimization under orthogonality constraints. Each landing step consists of the sum of a \emph{normal} component, which reduces infeasibility, and a \emph{tangent} component, which decreases the objective function. Our main contribution is the introduction of low-dimensional random \emph{sketch matrices} to reduce the computational cost of these directions. We consider both dense (Gaussian) and sparse (subsampling) sketch matrices, and show how they reduce the per-iteration cost while preserving convergence guarantees in expectation.
\end{abstract}


\section{Introduction}
For $n\geq p>0$, we consider the constrained minimization of a population risk,
\begin{equation}\label{eq:P}\tag{P}
\begin{aligned}
& \underset{X\in \R^{n\times p}}{\minimize} 
& & f(X)\coloneqq \mathbb{E}_{\xi\sim \mathcal{D}}[f_\xi(X)] \\
& \text{subject to}
&& X\T X = I_p,
\end{aligned}
\end{equation}
where $\xi \sim \mathcal{D}$ is a sample drawn from a distribution~$\mathcal{D}$, and each $f_\xi\colon \Rnp \to \R$ is a (possibly nonconvex) smooth loss with $L$-Lipchitz continuous gradient (i.e., $L$-smooth). 
The feasible set, called the (standard) Stiefel manifold, is the set of orthonormal $p$-frames in~$\mathbb{R}^n$,
\[
\St(n,p)\coloneqq\{ X\in\R^{n\times p }\,|\, X\T X= I_p \}.
\]   
Letting
\begin{align}\label{eq:h}
h\colon \Rnp \to \Sym(p)\colon X\mapsto X\transpose X - I_p,
\end{align}
the feasible set is $h\inv(0)$ and the infeasibility is measured by
\begin{align}\label{eq:psi}
	\mathcal{N}(X) \coloneqq \frac12 \fronorm{h(X)}^2.
\end{align}

Riemannian optimization methods 
maintain orthogonality using, e.g., a QR decomposition of an $n\times p$ matrix at every iteration~\cite{absil2008}. 
Landing methods provide an alternative to Riemannian optimization by allowing iterates to evolve away from the feasible set while still converging to it~\cite{vary2024Optimizationa,pmlr-v151-ablin22a, GaoVaryAblinAbsil2022, JMLR:v25:23-0451, pmlr-v195-schechtman23a,SiMalick26}. These methods exploit the local geometry of the constraint function~$h$: each update direction combines a normal component that reduces infeasibility with a tangent component that decreases the objective function~$f$. Recent work has highlighted close connections between the landing framework and sequential quadratic programming~\cite{goyens2026riemannian}. In particular, landing methods are based on first-order directions, which are commonly used in machine learning applications, while avoiding the delicate tuning of penalty parameters required by many traditional penalty approaches.

Yet, landing methods need matrix products computed in $\calO(np^2)$ that may greatly outweigh the cost of computing a stochastic gradient through backpropagation in modern deep learning applications. Therefore, we propose a stochastic version of the landing method, which uses low-dimensional random sketches to drastically reduce the size of the matrix products that form the landing iteration. Our method builds on an existing algorithm for optimization on the generalized Stiefel manifold with stochastic constraints~\cite{vary2024Optimizationa}, for which we design low-rank sketches to approximate the identity matrix. 

Other randomized optimization algorithms on the Stiefel manifold have been proposed, for example in \cite{FeiFengFan2025} and \cite{HanPoirionTakeda2025}. In \cite{FeiFengFan2025}, the randomization is performed over the columns, and in both works the iterates remain on the manifold, in contrast with the landing framework.

\cref{sec:preliminaries} starts by introducing the landing framework on the generalized Stiefel manifold. Then, \cref{sec:sketched-landing} introduces the sketched landing algorithm on the standard Stiefel manifold. Two sketching strategies are given in \cref{sec:sketches}. A variance-reduction modification of the algorithm is proposed in \cref{sec:svrg} based on the SVRG framework~\cite{johnson2013accelerating}. Finally, \cref{sec:convergence} establishes a convergence result and \cref{sec:numerics} presents numerical experiments.
\section{Preliminaries: geometry and the landing framework}
\label{sec:preliminaries}
Presenting the sketched landing method requires introducing optimization over the \emph{generalized} Stiefel manifold. Given a symmetric positive definite matrix $B\in \Rnn$, it is defined as
\[
\St^B(n,p) \coloneqq \{X\in \Rnp:\; X^\top B X = I_p\}.
\]
The corresponding constraint function is
\begin{equation*}
    h^B(X) \coloneqq X\T B X - I_p,
\end{equation*}
and the infeasibility measure is
\begin{equation*}
    \mathcal{N}^B(X) \coloneqq \dfrac{1}{2}\fronorm{h^B(X)}^2.
\end{equation*}
In this paper, we eventually restrict ourselves to the particular case $B=I_n$. 

We assume throughout that all iterates belong to the set $\Rnpstar$ of $n\times p$ full-rank matrices. 
For every $X\in \Rnpstar$, consider the set
\begin{equation*}
    \StXBnp = \left\{Y\in \Rnpstar:X\T B X = Y\T B Y \right\},
\end{equation*}
which denotes the level set of the function $h^B$ corresponding to the point $X$. 

The landing method relies on the observation that for every $X\in \Rnpstar$, the set $\StXBnp$ is a smooth manifold with the same dimension as $\StBnp$~\cite{goyens2026riemannian}. The tangent and normal space to $\StXBnp$ at $X\in \Rnpstar$ are given by
\begin{equation*}
    \begin{aligned}
        \rmT_X\StXBnp &= \ker \D h^B(X)\\
        \rmN_X \StXBnp &= \range \D h^B(X)^*.
    \end{aligned}
\end{equation*}
See~\cite[Section 7]{goyens2026riemannian} for explicit characterizations of the tangent and normal spaces.

\subsection{Landing method on $\St^B(n,p)$}

\label{sec:landing-baseline}
Like many optimization methods, the landing method is a general framework rather than a specific algorithm. For $t=0,1,2...$, it takes the form 
\begin{equation}\label{eq:landing}
    X_{t+1} = X_t + \alpha_t \left(\omega_{\mathrm{T}} d_{\mathrm{T}}(X_t) + \omega_{\mathrm{N}} d_{\mathrm{N}}(X_t)\right),
\end{equation}
where $\omega_{\mathrm{T}},\omega_{\mathrm{N}} \geq 0$ are parameters, $\alpha_t > 0$ is an adaptive step size, $d_{\mathrm{T}}(X_t)\in \rmT_{X_t} \St_{X_t}^B(n,p)$ is a \emph{tangent descent component} for $f$, and $d_{\mathrm{N}}(X_t) \in \rmN_X \St_{X_t}^B(n,p)$ is a \emph{normal component} reducing the infeasibility measure. Dropping the index $t$, popular choices for $d_{\mathrm{T}}(X)$ and $d_{\mathrm{N}}(X)$ are
\begin{equation}\label{eq:landing_generalized_stiefel}
\begin{aligned}
    d_{\mathrm{T}}(X) &= - \grad^B f(X), \\
   \text{and}\ d_{\mathrm{N}}(X) &= - \nabla \mathcal{N}^B(X),
\end{aligned}
\end{equation}
where $\grad^B f(X)$ is the Riemannian gradient on $\StBnp$ of $f$ at $X$. It is a descent direction for $f$ in the tangent space to the current infeasibility level set, and it is given by
\begin{align*}
	\grad^B f(X) &= 2\skew\left( \nabla f(X) X\T B \right) B X,
\end{align*}
where $\skew(A) \coloneqq (A-A\T)/2$. This constrained gradient corresponds to an extension of the canonical metric, see~\cite{vary2024Optimizationa}.

The normal component $d_{\mathrm{N}}(X)$ is the unconstrained gradient of the infeasibility measure (in the Euclidean metric), given by
\begin{equation*}
\nabla \mathcal{N}^B(X) = 2BX(X\T BX - I_p).
\end{equation*}

For optimization on $\Stnp$, the terms in~\eqref{eq:landing_generalized_stiefel} reduce to 
\begin{equation}\label{eq:landing_stiefel}
\begin{aligned}
	\grad f(X) &= 2\skew\left( \nabla f(X) X\T \right)X,  \\
 \nabla \mathcal{N}(X) &= 2X(X^\top X-I_p). 
\end{aligned}
\end{equation}
Hence we define 
\begin{equation}\label{eq:unsketched_landing}
\begin{aligned}
    d_\mathrm{T}^{\mathrm{landing}}(X) &= - \grad f(X),\\  d_\mathrm{N}^{\mathrm{landing}}(X) &= - \nabla \mathcal{N}(X).
    \end{aligned}
\end{equation}

\subsection{Stochastic landing method on $\St^B(n,p)$}
In~\cite{vary2024Optimizationa}, the authors introduce a stochastic landing method for the problem 
\[\min_{X} f(X) \ \text{s.~t.}\ X\in \StBnp,\]
where the matrix $B\in \Rnn$ is unavailable but unbiased random estimates $B_\zeta \in \Rnn$ can be obtained, i.e., $\mathbb{E}_\zeta [B_\zeta]=B$.
Consider two independent estimates $B_\zeta$ and $B_{\zetaprime}$ of $B$, as well as an unbiased stochastic estimate $\nabla f_\xi(X)$ of the gradient $\nabla f(X)$. The update step reads
\begin{align}
    \nonumber
	X_{t+1} = X_t - \alpha_t \big(&\omega_\mathrm{T}\grad^{B,\zeta,\zeta'} f_\xi(X_t)\\
 \label{eq:landing_vary}
 &+ \omega_\mathrm{N} \nabla\mathcal{N}^{B,\zeta,\zeta'}(X_t)\big),
\end{align}
where 
\begin{equation*}
 \nabla\mathcal{N}^{B,\zeta,\zeta'}(X) = 2 B_{\zeta'} X(X\transpose B_\zeta X - I_p),
\end{equation*}
and 
\begin{equation*}
	\grad^{B,\zeta,\zeta'} f_\xi(X) = 2 \skew(\nabla f_\xi(X)X\transpose B_\zeta)B_{\zeta'}X.
\end{equation*}

\section{Sketched landing on $\Stnp$}
\label{sec:sketched-landing}
This section shows that the framework of stochastic optimization on the generalized Stiefel manifold can be used to design a randomized landing algorithm for optimization on the standard Stiefel manifold. The matrix $B=I_n$ is estimated using unbiased low-rank \emph{sketch matrices}. This allows reducing the dimensions of the matrix products that appear in~\eqref{eq:landing_stiefel}.
Indeed, consider random matrices of the form
\begin{equation}\label{eq:sketch}
B_\zeta = S_\zeta S_\zeta\transpose,\qquad S_\zeta\in \R^{n\times k},\qquad k\ll p\leq n,
\end{equation}
where the sampling distribution is unbiased:
\begin{equation}\label{eq:unbiased}
\mathbb{E}_\zeta[B_\zeta]=\mathbb{E}_\zeta[S_\zeta S_\zeta^\top]=I_n.
\end{equation}
The size $k$ is the \emph{sketch dimension}. In~\cref{sec:sketches}, we propose two options for generating the sketch matrices~$S_\zeta$.

\subsection{Sketched normal and tangent components}
\label{sec:sketched-directions}
Consider two independent sketches $S_{\zeta}$ and $S_{\zeta'}$ sampled from the same unbiased distribution. The sketched directions are defined  using~\eqref{eq:sketch} in the landing directions for the generalized Stiefel manifold~\eqref{eq:landing_vary}. The sketched tangent component is
\begin{align*}
\dTsketch(X) & =- \grad^{B,\zeta,\zeta'}  f_\xi(X)\\
&=-2 \skew\left(\nabla f_\xi(X)X\transpose B_\zeta\right)B_{\zeta'}X \\
&= -2 \skew \left(\nabla f_\xi(X) X\T S_\zeta S_\zeta\T\right) S_{\zetaprime} 
	S_{\zetaprime}\T X.
 \end{align*}
By grouping the operations to avoid building an $n\times n$ matrix, we obtain
 \begin{align}
\nonumber
&\dTsketch(X) = -\left[\nabla f_\xi(X)\left( \left( S_\zeta\T X\right)\T (S_\zeta\T S_{\zetaprime})\right)\right] \left(S_{\zetaprime}\T X\right)\\
 \label{eq:sketch-tangent}
&\quad\ + S_\zeta \left[ \left(\left(S_\zeta\T X\right) \left(\nabla f_\xi(X)\T S_{\zetaprime}\right)\right) 
\left(S_{\zetaprime}\T X\right)\right].
 \end{align}
In order to obtain the desired speed-up in terms of computational complexity, it is crucial to compute the matrix products as proposed by the parentheses. Indeed, associating the matrix products in the wrong order may yield a very suboptimal operation count. This is explained in more details in \cref{lem:count_dense,lem:count_sparse}. 

The sketched normal component is
\begin{align}
\nonumber
d_{\mathrm{N}}^{\mathrm{sketch}}(X) &=  -\nabla\mathcal{N}^{B,\zeta,\zeta'}(X) \\
\nonumber
&= -2 B_{\zeta'} X\left(X\transpose B_\zeta X - I_p \right)	\\
\nonumber
 &= -2 S_\zetaprime S_\zetaprime\T X\left(X\transpose S_\zeta S_\zeta\T X - I_p \right)\\
 \nonumber
 &=-2 S_\zetaprime \left[\left((S_\zetaprime\T X)  (S_\zeta\T X)\T\right) (S_\zeta^\top X)\right] \\
 \label{eq:sketch-normal}
 &\quad + 2S_\zetaprime (S_\zetaprime\T X).
 \end{align}

In addition to the cost of computing the tangent and normal components, one should take into account the cost of evaluating the gradient $\nabla f(X)$ or the stochastic estimate $\nabla f_\xi(X)$. In particular, for problems where evaluating the gradient requires $\calO(n p^2)$ operations or more, dense sketch matrices may not yield any computational improvement. However, if the sketch matrix is sparse, the sparsity pattern should be taken into account to reduce the cost of evaluating  $\nabla f(X)$ or $\nabla f_\xi(X)$. Additional details about computational complexity can be found at \url{https://github.com/flgoyens/SketchedLanding}.

\subsection{The sketched landing algorithm}
\label{sec:alg-sketch-landing}

A pseudo-code of the sketched landing algorithm is given in~\cref{algo:sketchlanding}.

\begin{algorithm}[h]
\caption{Sketched landing}
\label{algo:sketchlanding}
\begin{algorithmic}[1]
\Require Initial $X_0\in \Rnpstar$,  step sizes $\{\alpha_t >0\}_{t=0}^T$, penalty weights $\omega_\mathrm{N},\omega_\mathrm{T}>0$ and sketch size $k$.
\For{$t=0,1,\dots, T$}
\State Sample sketches $S_{\zeta_t},S_{\zeta'_t}\in \R^{n\times k}$. 
\State (Optional) Sample $\xi_t$ and compute $\nabla f_{\xi_t}(X_t)$.
\State Compute $\dTsketch(X_t)$ and $\dNsketch(X_t)$. 
\State $d^{\mathrm{sketch}}(X_t) = \omega_{\mathrm{T}} \dTsketch(X_t) + \omega_{\mathrm{N}} \dNsketch(X_t) $.
\State Set $X_{t+1} = X_t + \alpha_t d^{\mathrm{sketch}}(X_t) $.
\EndFor
\end{algorithmic}
\end{algorithm}

\section{Types of sketches}
\label{sec:sketches}
This section proposes two different unbiased distributions for sampling the sketch matrices of \cref{algo:sketchlanding}. In the subsequent flop counts, the availability of $\nabla f_\xi(X)$ is always assumed.
\subsection{Dense sketches}\label{sec:dense}
Let us first show in~\cref{prop:expectation_sketch} that scaled Haar-distributed matrices on $\St(n,k)$~\cite{Stewart80} (i.e., uniformly distributed) are suitable for sampling unbiased estimates of $I_n$ as in~\eqref{eq:sketch}.

\begin{lemma}\label{prop:expectation_sketch}
 Let $S_\zeta = \sqrt{\frac{n}{k}} R_\zeta$, where $R_\zeta$ is Haar distributed on $\St(n,k)$. Then, 
 \begin{align}
 	\bbE_\zeta\left[S_\zeta  S_\zeta\T \right] &= I_n.
 \end{align}
 \end{lemma}
 \begin{proof}
 By definition, if $R_\zeta$ is Haar-distributed, for all $Q\in\mathrm{SO}(n)$, then $R_\zeta \sim QR_\zeta$~\cite{Stewart80}. In particular, this yields
 \begin{align}
 \nonumber
 \bbE_\zeta\left[R_\zeta R_\zeta^\top \right] &= \bbE_\zeta\left[QR_\zeta R_\zeta^\top Q^\top\right]\\
 \label{eq:variance}
 &= Q\bbE_\zeta\left[R_\zeta R_\zeta^\top\right]Q^\top.
 \end{align}
 Since \eqref{eq:variance} holds for all $Q\in\mathrm{SO}(n)$, it implies that $$\bbE_\zeta[R_\zeta R_\zeta^\top] = \kappa I_n,$$ for some $\kappa\in\mathbb{R}$. Moreover, since $\mathrm{Tr}(R_\zeta R_\zeta^\top)=k$, we have $\kappa= \frac{k}{n}$. By defining $S_\zeta = \sqrt{\frac{n}{k}} R_\zeta$, the claim follows.
 \end{proof}
A method for generating sketch matrices satisfying \cref{prop:expectation_sketch} is to orthogonalize and scale matrices with normally distributed entries. It takes the simple form
$$\text{(Dense sketch)} \quad S_\zeta = \sqrt{\frac{n}{k}}\mathtt{qf}\left(\mathtt{randn}(n,k)\right),$$ 
where $\mathtt{qf}$ denotes the orthogonal factor of a (thin) QR decomposition where the upper triangular factor has nonnegative diagonal entries. Sampling a sketch requires $\calO(nk^2)$ flops in this case.

Moreover in this setting, computing the product $S_\zeta^\top X$ demands $\calO(npk)$ flops. It is verified in \cref{lem:count_dense} that computing $\dTsketch(X)$ and $\dNsketch(X)$ with the suggested association of matrix products reduces to $\calO(npk)$ flops instead of $\calO(np^2)$ flops for the unsketched landing algorithm from \cite{vary2024Optimizationa}.
\begin{lemma}\label{lem:count_dense}
    Computing $d^{\mathrm{sketch}}(X)$ in \cref{algo:sketchlanding} using dense sketches requires $\calO(npk)$ flops.
\end{lemma}
\begin{proof}
    Both products $S_\zeta^\top X$ and $S_\zetaprime^\top X$ require $\calO(npk)$ flops. The product $S_\zeta ^\top S_\zetaprime$ requires $\calO(nk^2)$ flops. Then by \eqref{eq:sketch-tangent}, it is straightforward but tedious to verify that computing the matrix products as proposed by the parentheses yields $\calO(npk)$ flops for computing $\dTsketch(X)$. Moreover, computing \eqref{eq:sketch-normal} as suggested by the parentheses also requires $\calO(npk)$ flops to obtain $\dNsketch(X)$.
\end{proof}

\subsection{Sparse sketches}
It may be argued that reducing the complexity to $\mathcal{O}(npk)$ flops per iteration is still insufficient. An even more computationally attractive sketching strategy is obtained by uniformly sampling $k$ columns of the identity matrix $I_n$ and scaling them by a factor $\sqrt{\frac{n}{k}}$.
\[
\text{(Sparse sketch)}\quad S_\zeta = \sqrt{\frac{n}{k}} 
\revision{I_{n}[1:n, \zeta]},
\]
where $\zeta$ is a list of $k$ distinct integers uniformly sampled between $1$ and $n$.

From a computational point of view, multiplying from the left by a sparse sketch is a selection and scaling of $k$ rows among $n$. Thus the product $S_\zeta^\top X$ requires only $\calO(pk)$ flops. In consequence, the matrix $S_\zeta S_\zeta^\top$ is an $n\times n$ diagonal matrix with only $k$ nonzero diagonal entries that are equal to $\frac{n}{k}$. Moreover, as shown in \cref{lem:count_sparse}, the matrix $S_\zeta^\top S_\zetaprime$ from \eqref{eq:sketch-tangent} has only $\frac{k^2}{n}$ nonzero entries in expectation. Moreover, for $n\gg k$, $$\mathbb{P}[S_\zeta^\top S_\zetaprime = 0] = \frac{(n-k)!^2}{n!(n-2k)!}\approx e^\frac{-k^2}{n}.$$ The consequence for the computation of $\dTsketch(X)$ is important since it allows further reduction of the complexity from $\calO(npk)$ to $\calO(pk^2)$. Surprisingly, the sparse sketches allow to make the factor $n$ completely disappear from the expected cost of evaluating $d^{\mathrm{sketch}}(X)$. 
 \begin{lemma}\label{lem:count_sparse}
    Computing $d^{\mathrm{sketch}}(X)$ in \cref{algo:sketchlanding} using sparse sketches requires $\calO(pk^2)$ flops in expectation.
\end{lemma}
\begin{proof}
    The expression of the sketched tangent component of the landing direction is given by
 \begin{align*}
\dTsketch(X) &= -\left[\nabla f_\xi(X)\left( \left( S_\zeta\T X\right)\T (S_\zeta\T S_{\zetaprime})\right)\right] \left(S_{\zetaprime}\T X\right) \\
&+ S_\zeta \left[ \left(\left(S_\zeta\T X\right) \left(\nabla f_\xi(X)\T S_{\zetaprime}\right)\right) 
\left(S_{\zetaprime}\T X\right)\right].
 \end{align*}
 For the sparse sketches, computing  $S_\zeta^\top X$ and $S_\zetaprime^\top X$ reduces to $\calO(pk)$ flops. The $k\times k$ matrix $S_\zeta^\top S_\zetaprime$ is sparse with $|\zeta\cap\zetaprime|$ nonzero entries, and at most one nonzero entry per row and per column. Moreover,  $S_\zeta\T S_{\zetaprime}=0$ with probability $\frac{(n-k)!^2}{n!(n-2k)!}$, in which case the computation reduces to 
 \begin{equation*}
\dTsketch(X) =  S_\zeta \left[ \left(\left(S_\zeta\T X\right) \left(\nabla f_\xi(X)\T S_{\zetaprime}\right)\right) 
\left(S_{\zetaprime}\T X\right)\right].
 \end{equation*}
 In this case, it is simple to verify that $\dTsketch(X)$ is computed with $\mathcal{O}(pk^2)$ flops using the association suggested by the parenthesis. We recall that $X$ and $\nabla f_\xi(X)$ are $n\times p$ matrices. 
 
Otherwise, it is verified in \url{https://github.com/flgoyens/SketchedLanding} that
    \[
    \mathbb{E}_{\zeta,\zetaprime}[|\zeta\cap\zetaprime|] = \frac{k^2}{n}\  \text{and}\ S_\zeta^\top S_\zetaprime = P_\mathrm{l}DP_\mathrm{r},
    \]
    where $D$ is diagonal with $|\zeta\cap\zetaprime|$ positive diagonal entries and $P_\mathrm{l},P_\mathrm{r}$ are permutation matrices. Then, it follows that
 \begin{align*}
     &\left[\nabla f_\xi(X)\left( \left( S_\zeta\T X\right)\T (S_\zeta\T S_{\zetaprime})\right)\right] \left(S_{\zetaprime}\T X\right) \\
     &= \left[\nabla f_\xi(X)\left( \left( S_\zeta\T X\right)\T ( P_\mathrm{l}\sqrt{D})\right)\right] \left(\sqrt{D}P_\mathrm{r}\left(S_{\zetaprime}\T X\right)\right).
 \end{align*}
Since $D$ has $\frac{k^2}{n}$ nonzero diagonal entries in expectation, the product $\left( S_\zeta\T X\right)\T ( P_\mathrm{l}\sqrt{D})$ has $\frac{k^2}{n}$ nonzero columns in expectation. The same holds for $\sqrt{D}P_\mathrm{r}\left(S_{\zetaprime}\T X\right)$ with the rows. The matrix $\nabla f_\xi(X)$ is $n\times p$ and the product $\nabla f_\xi(X)\left( \left( S_\zeta\T X\right)\T ( P_\mathrm{l}\sqrt{D})\right)$ costs $\mathcal{O}(np\frac{k^2}{n}) = \mathcal{O}(pk^2)$ flops in expectation. Finally the multiplication by the right by $\sqrt{D}P_\mathrm{r}\left(S_{\zetaprime}\T X\right)$ also costs $\mathcal{O}(pk^2)$ flops in expectation.

For the sketched normal component of the landing term, we have
\begin{align*}
    d_{\mathrm{N}}^{\mathrm{sketch}}(X) &= -2 S_\zetaprime \left[\left((S_\zetaprime\T X)  (S_\zeta\T X)\T\right) (S_\zeta^\top X)\right] \\
    &+ 2S_\zetaprime (S_\zetaprime\T X).
\end{align*}
The complexity of $\mathcal{O}(pk^2)$ flops follows directly. We have thus shown that computing the landing direction from $X, \nabla f_\xi(X), S_\zeta$ and $S_\zetaprime$ costs $\mathcal{O}(pk^2)$ flops in expectation.

\end{proof}
\begin{table}
\centering
\caption{Computational cost of computing the tangent and normal components of the landing direction, without the cost of evaluating $\nabla f(X)$ or $\nabla f_\xi(X)$. }
{\normalsize
\begin{tabular}{|c|c|}
\hline
Formula & Expected flops \\
  \hline
  \hline
 $d(X)$ & $\calO(n p^2)$ \\
  \hline
 $d^\mathrm{sketch}(X)$ (dense sketch) & $\calO(npk)$ \\
\hline
  $d^\mathrm{sketch}(X)$ (sparse sketch) & $\calO(pk^2)$ \\
\hline
\end{tabular}}

\end{table}

\section{Variance reduction}\label{sec:svrg}
To reduce the variance introduced by random sketches, we propose a variance reduction technique, inspired by the SVRG framework used for finite-sum minimization~\cite{johnson2013accelerating}. For some parameter $m>0$, every $m$ iteration, we compute an \emph{unsketched} landing step~\eqref{eq:unsketched_landing}. This unsketched step is used to improve the directions during the next $m$ iterations, after which a new unsketched step is computed (\cref{algo:svrg_sketch_landing}).

\begin{algorithm}
\caption{SVRG sketched landing}
\begin{algorithmic}[1]
\State \textbf{Given:} $X_0\in \Rnpstar$, step sizes $\{\alpha_t >0\}_{t=0}^T$, $\omega_\mathrm{T},\omega_\mathrm{N}>0$, $m \in \mathbb{N}_*$ and sketch size $k$.
\State \textbf{For} $t=0,1,\dots,T$
\State $Y = X_t$ if $t$ is a multiple of $m$
\State Generate sketches $S_\zeta,S_\zetaprime \in \Rnk$ 
\State $d_{\mathrm{T}}^{\mathrm{svrg}}(X_t) = d^{\mathrm{landing}}_\mathrm{T}(Y) - d^{\mathrm{sketch}}_\mathrm{T}(Y) + d^{\mathrm{sketch}}_\mathrm{T}(X_t)$
 \State $d_{\mathrm{N}}^{\mathrm{svrg}}(X_t) = d^{\mathrm{landing}}_\mathrm{N}(Y) - d^{\mathrm{sketch}}_\mathrm{N}(Y) + d^{\mathrm{sketch}}_\mathrm{N}(X_t)$ 
\State $X_{t+1} = X_t + \alpha_t \left(\omega_{\mathrm{T}} d_{\mathrm{T}}^{\mathrm{svrg}}(X_i) + \omega_{\mathrm{N}} d_{\mathrm{N}}^{\mathrm{svrg}}(X_t) \right)$
\State \textbf{End for}
\end{algorithmic}
\label{algo:svrg_sketch_landing}
\end{algorithm}

The SVRG formula features the difference of constrained Riemannian gradients computed at different points. This is well defined since those tangent vectors belong to the ambient space $\Rnp$. An expensive alternative would be to transport the tangent vectors to the tangent space at the current point $X_t$.

\section{Convergence analysis}\label{sec:convergence}
We show, under common assumptions, that \cref{algo:sketchlanding,algo:svrg_sketch_landing} converge in expectation towards a critical point of~\eqref{eq:P} using a sequence of decreasing step sizes.

In~\cite{vary2024Optimizationa}, the authors show convergence in expectation for the stochastic landing iteration~\eqref{eq:landing_vary} for optimization on the generalized Stiefel manifold. The sketched landing is an instance of this iteration, for the particular case where $B$ is the identity and the unbiased estimators $B_{\zeta}$ are computed from sketch matrices. 

Convergence proofs for landing methods rely on the existence of a \emph{safe region}, which contains all iterates by assumption. The safe region~\cite[page 4]{goyens2024computing} is defined for some constant $r\in(0,1)$ as
\begin{equation}
    \mathcal{M}^r= \left\{ X\in \Rnp: \norm{h(X)}_\mathrm{F}\leq r\right\},
\end{equation}
and satisfies
\begin{equation}
    \sigmamin(\D h(X)) \geq 2 \sqrt{1-r}>0, \quad \text{for all } X\in  \mathcal{M}^r.
\end{equation}
The safe region is a strict subset of $\Rnpstar$, where the landing method is well defined. 

The following convergence result is consistent with rates for stochastic gradient methods on nonconvex problems. For an expanded statement and formal proof, see~\url{https://github.com/flgoyens/SketchedLanding}.
\begin{theorem}
    Consider the sketched landing~\cref{algo:sketchlanding} and the SVRG version~\cref{algo:svrg_sketch_landing} with dense or sparse sketches as defined above, and step sizes satisfying $\alpha_t = \alpha_0/\sqrt{t+1}$. Further assume that each $f_\xi$ is $L$-smooth and stochastic gradients $\nabla f_\xi(x)$ are unbiased estimates of $\nabla f(x)$ with bounded variance. Provided that the segment connecting all iterates remains in the safe region $\mathcal{M}^r$ with probability one, then the iterates of~\cref{algo:sketchlanding,algo:svrg_sketch_landing}  satisfy
    \begin{equation*}
        \inf_{t\leq T} \mathbb{E}\left[ \norm{\grad f(X_t)}_\mathrm{F}^2 \right ] \leq \tilde\calO\left( \dfrac{1}{\sqrt{T}}\right)
    \end{equation*}
    and 
     \begin{equation*}
        \inf_{t\leq T} \mathbb{E}\left[ \norm{h(X_t)}_\mathrm{F}^2 \right ] \leq \tilde\calO\left( \dfrac{1}{\sqrt{T}}\right)
    \end{equation*}
    for $\alpha_0$ small enough.
\end{theorem}
\begin{proof}
    The result follows from~\cite[Thm. 2.9]{vary2024Optimizationa}, since the directions of~\cref{algo:sketchlanding,algo:svrg_sketch_landing} are unbiased estimates of the landing field~\eqref{eq:landing_stiefel} with bounded variance. 
\end{proof}

\section{Numerical results}\label{sec:numerics}

\cref{fig:2d} illustrates the sketched landing dynamics for the minimization of a quadratic function on the unit circle. The sketched directions provide a noisy approximation of the deterministic trajectory.

\begin{figure}[h]
 
 \hspace{-0.5cm}
  \includegraphics[width=0.55\textwidth]{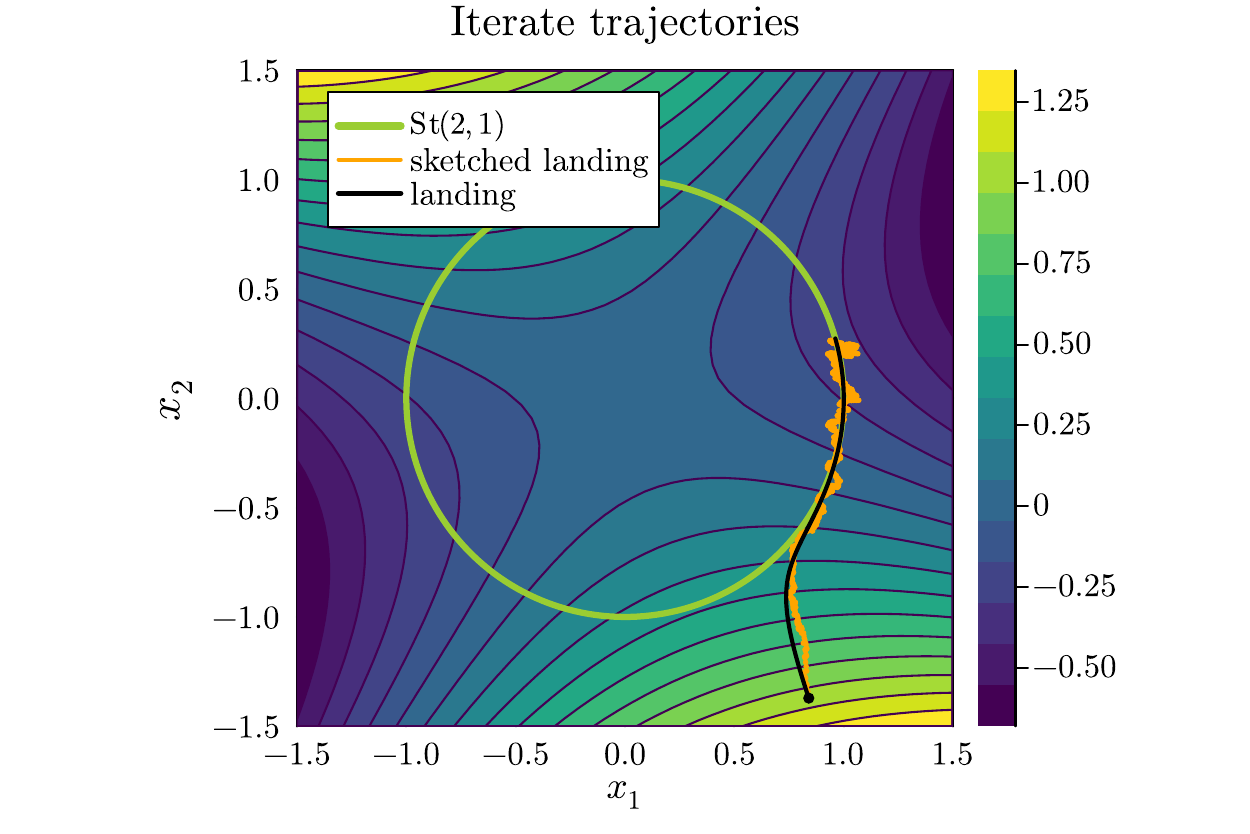}
  
\caption{Iterates of the landing and sketched landing from \cref{algo:sketchlanding} in $\R^2$.}

\label{fig:2d}
\end{figure} 

We also report preliminary numerical results in moderate dimensions. The sketched landing framework is intended to yield significant computational gains in high-dimensional regimes, which are beyond the scope of the present manuscript. 

First note that the tuning and adaptivity of step sizes is a critical issue for the performance of landing methods, which we do not investigate here: all experiments are done with constant step sizes. Preliminary experiments indicate that deterministic directions may allow to take larger step sizes than sketched directions and still maintain convergence.

\begin{figure}[!htb]
\centering  
  \includegraphics[width=0.5\textwidth]{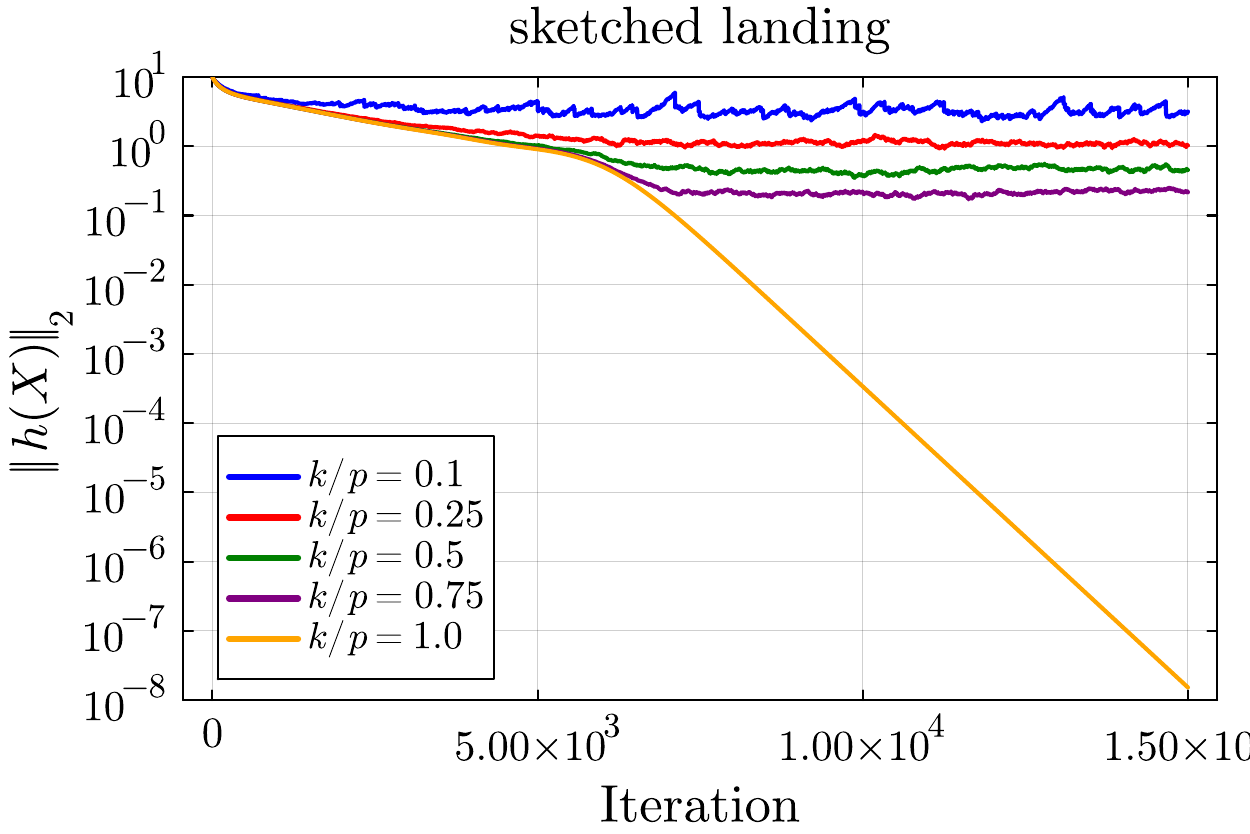}

\caption{Orthogonalization ($f=0)$, $n=p=100$, $\alpha = 10^{-3}$.}
\label{fig:plain}
\end{figure}

\begin{figure}[!htb]
\centering  
  \includegraphics[width=0.5\textwidth]{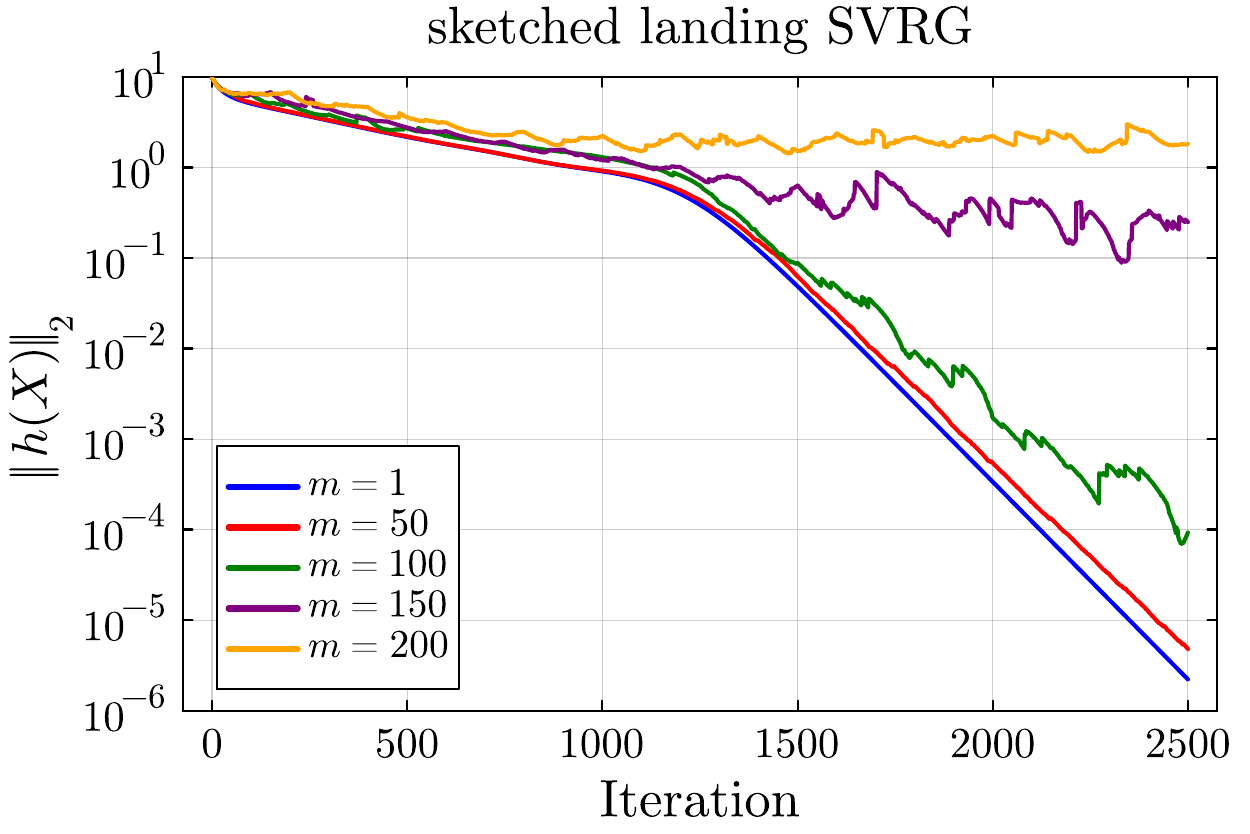}

\caption{Orthogonalization ($f=0)$: sketched landing SVRG with $\alpha = 5\cdot 10^{-3}$, $n=p=100$, $k=10$.}
\label{fig:svrgT}
\end{figure} 

In the experiments of \cref{fig:plain}, we start by considering the problem of orthogonalizing a matrix (i.e., $f=0$). The plot shows that the sketched landing without variance reduction tends to plateau around a noise level that depends on the sketch dimension $k$; whereas \cref{fig:svrgT} shows that the SVRG version can converge to high accuracy, even for small sketch dimension $k$. The SVRG method behaves essentially like the deterministic landing for appropriate values of $m$; and performance can degrade if $m$ becomes too large.

\begin{figure}[h!]
\centering  

 \includegraphics[width=0.5\textwidth]{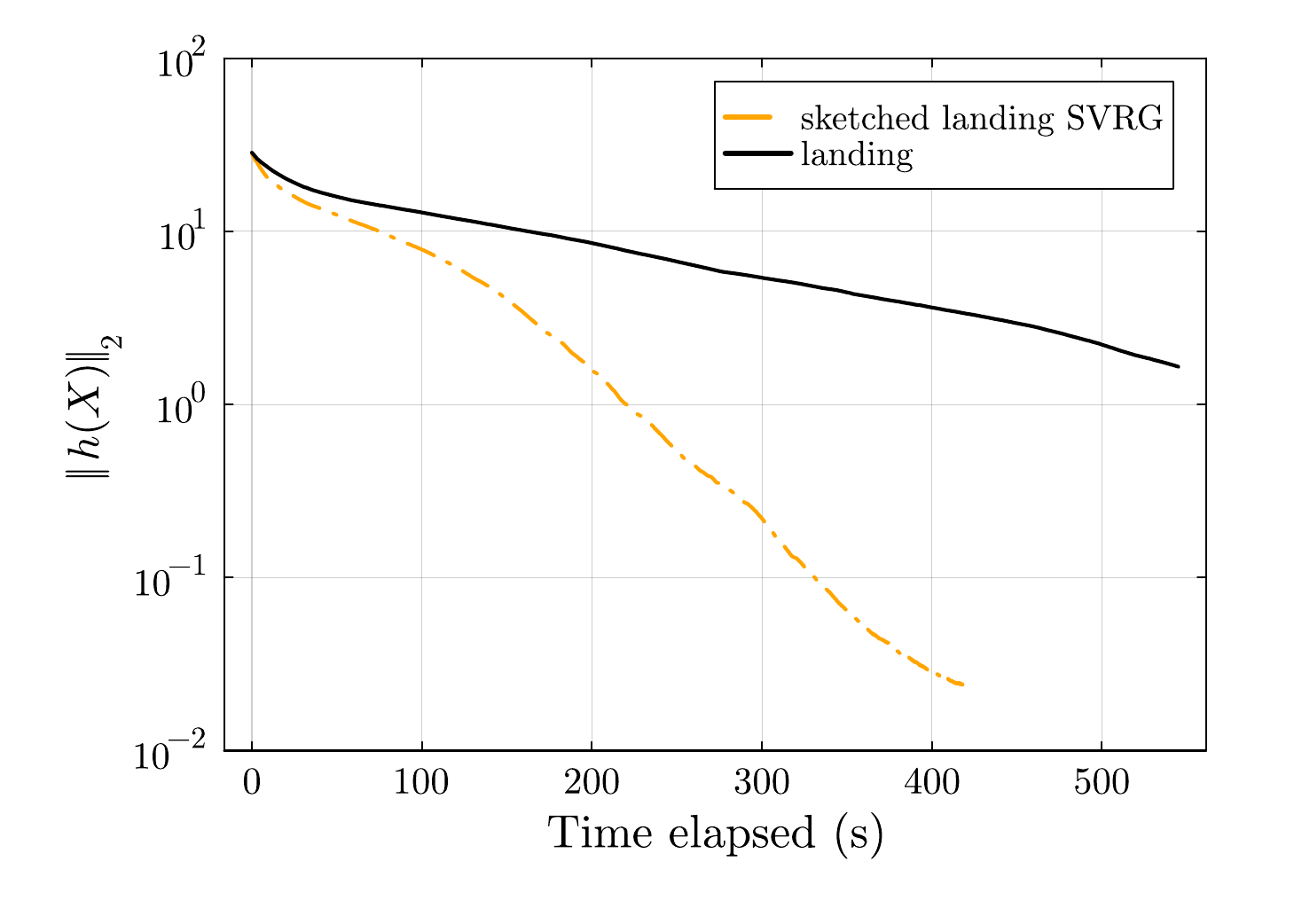}

\includegraphics[width=0.5\textwidth]{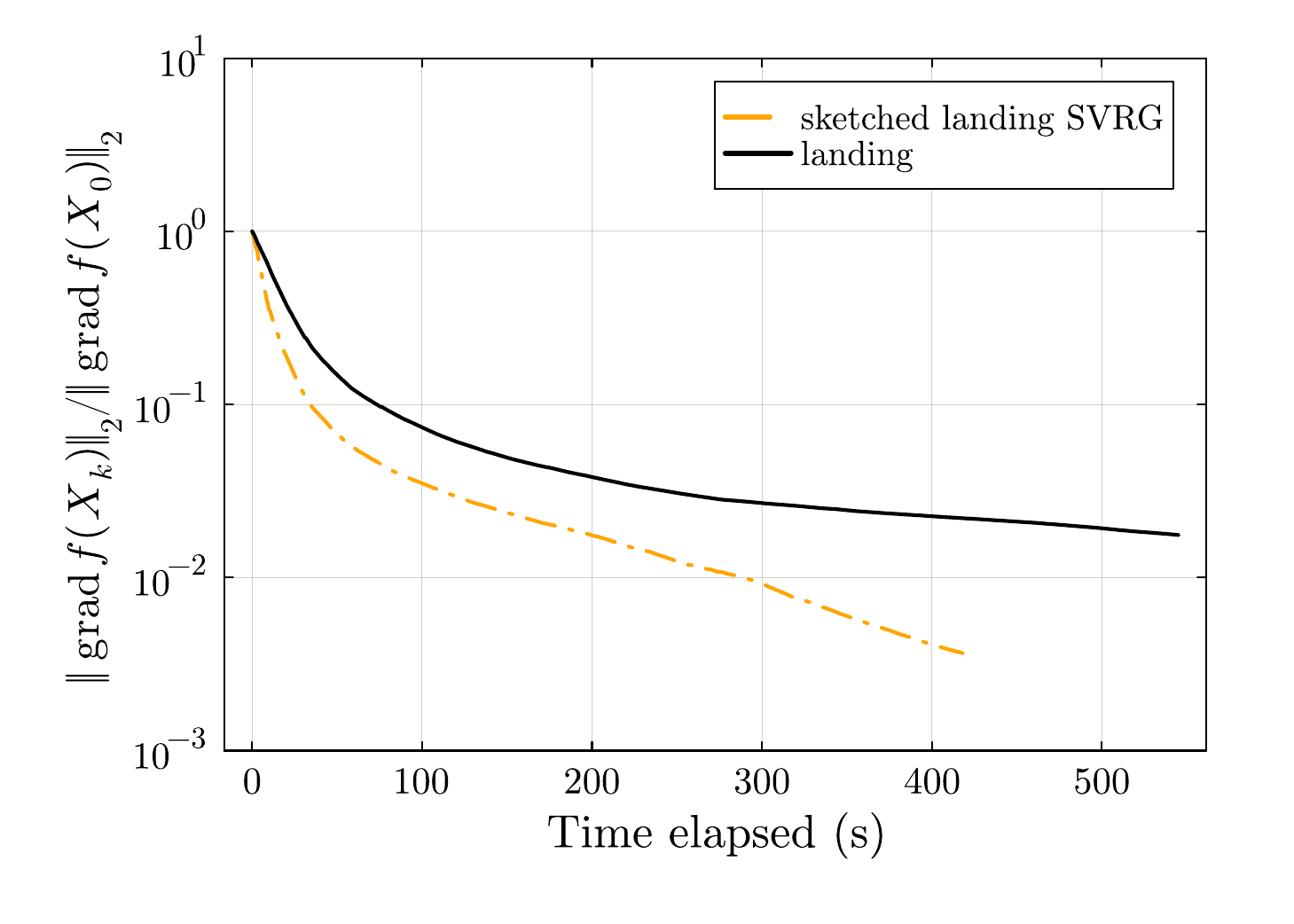}

\caption{Rayleigh quotient: sketched landing SVRG vs landing with $\alpha = 10^{-3}$, $n=1000$, $p=900$, $k=50$, and $m=20$.}

\label{fig:comparison}
\end{figure} 

Finally,~\cref{fig:comparison} shows that the sketched landing performs better than the deterministic landing in a case where both methods use the same step size to optimize the Rayleigh quotient $\mathrm{tr}\left(X \T A X\right)$ where $A$ is sparse and symmetric. 

\section{Conclusion}
We have introduced the sketched landing method, a randomized variant of the landing framework for optimization under orthogonality constraints. The method reduces the $O(np^2)$ computational complexity of the matrix products in the tangent and normal components of the landing field by using low-dimensional random sketches. Dense Gaussian sketches reduce the per-iteration cost to $O(npk)$, while sparse subsampling sketches reduce it further to $O(pk^2)$ in expectation. An SVRG-based variance-reduction scheme mitigates the noise floor introduced by sketching. Future work includes a high-performance implementation of the sketched landing algorithm, as well as numerical experiments on large-scale problems.







\bibliographystyle{IEEEtran}
\bibliography{mynicenewbib}
\end{document}